\documentclass[12pt, notitlepage]{amsart}
\usepackage{latexsym, amsfonts, amsmath, amssymb, amsthm, cite}
\usepackage{enumerate}

\newtheorem{proposition}{Proposition}
\newtheorem{theorem}{Theorem}

\newtheorem{conjecture}{Conjecture}

\usepackage{graphicx}
\usepackage{mathrsfs}

\title{Zeros in the character table of the symmetric group}

\author{Sarah Peluse and Kannan Soundararajan}

\address{Department of Mathematics, Stanford University, Stanford, CA, USA}
\email{speluse@stanford.edu}
\email{ksound@stanford.edu}
\dedicatory{To Trevor Wooley on the occasion of his sixty first birthday}

\thanks{The work of the first author was partially suppored by NSF
  grant DMS-2516641 and a Sloan Research Fellowship and the work of
  the second author was partially supported by NSF grant DMS-2100933.}

\begin{document}

\begin{abstract}
  Computations of Miller and Scheinerman suggest that the vast
  majority of the zeros appearing in the character table of the
  symmetric group are of a certain special type. While we cannot prove
  this, we resolve a conjecture arising in their paper
  concerning these zeros, and address a related question of
  Stanley.
\end{abstract}

\maketitle

  \section{Introduction} 
  
  \noindent In \cite{PS22} we showed that almost all entries in the character table of the symmetric group $S_N$ are multiples of any given prime number $p$, and in \cite{PS25} extended this to divisibility by prime powers.   This resolved conjectures of Miller \cite{Miller19}.   It is conceivable that these divisibility results hold because most of the character values are in fact zero, as is the case for high rank finite simple groups of Lie type~\cite{LarsenMiller2020}. This possibility has not been ruled out, but the numerical evidence (first generated by Miller \cite{Miller19}, who computed the character table for $N\le 38$) suggests that the proportion of zeros in the character table tends to zero with $N$.   More recent large-scale Monte Carlo simulations by Miller and Scheinerman \cite{MS25} suggest much more refined conjectures concerning the zeros in the character table.    In this paper we address a conjecture arising in their work, and also a related question of Stanley \cite{StanleyMO}; the problem of determining an asymptotic for the number of zeros in the character table remains open, but we can determine the number of zeros produced by three well-known criteria for guaranteeing zeros.   
 
  In order to state these conjectures fluidly, we review some notation and facts concerning the character table of the symmetric group.  Let $\lambda$ and $\mu$ denote partitions of the integer $N$.   We denote by $\chi^\lambda_\mu$ the value of the irreducible character of $S_N$ corresponding to the
  partition $\lambda$ evaluated on the conjugacy class of permutations in $S_N$ with cycle type corresponding to the partition $\mu$. It is an
  immediate consequence of the Murnaghan--Nakayama rule
  (see~\cite[Theorem 2.4.7]{JamesKerber1981}) that if $\mu$ has a part
  of size $t$ and $\lambda$ is a $t$-core, 
  then $\chi^\lambda_\mu=0$. Miller and Scheinerman call the pair of partitions $(\lambda,\mu)$ corresponding to a zero $\chi_{\mu}^{\lambda}=0$ in the character table 
  a \textit{zero of type II} if $\lambda$ is a $t$-core for some part size $t$ of $\mu$.   A pair $(\lambda,\mu)$ corresponds to a \textit{zero of type I} if
  $\lambda$ is a $\mu_1$-core, where $\mu_1$ denotes the largest part of   $\mu$. Clearly, type I zeros are a proper subset of type II
  zeros.

  The computations of Miller and Scheinerman suggest that most of the zeros appearing in the character table of $S_N$ are of
  type II, and that almost all type II zeros are type I zeros. After
  seeing the first version of their preprint, we pointed out to them
  that it is easy to show that the proportion of type I zeros in the
  character table of $S_N$ is asymptotically
  $\frac{2}{\log{N}}$. Based on this and their computational data,
  Miller and Scheinerman made the following conjecture.
  
  \begin{conjecture}[Miller and Scheinerman]
    The proportion of entries equal to zero in the character table of
    $S_N$ is asymptotically $\frac{2}{\log{N}}$.
  \end{conjecture}
  
  There is a third sufficient condition for $\chi_{\mu}^{\lambda}$ to equal zero, which includes all the zeros of type II.   For any natural number $t$, let ${\mathcal H}_t(\lambda)$ denote 
  the number of hook lengths in the Young diagram of $\lambda$ that are multiples of $t$.   Further, define 
  $$
  {\mathcal P}_t(\mu) = \frac 1t \sum_{t| \mu_j} \mu_j, 
  $$ 
  so that $t{\mathcal P}_t(\mu)$ is the sum of the parts in $\mu$ that are multiples of $t$.   Then a criterion of Stanley \cite{Stanleypaper} (see Lemma 7.4) states that 
  \begin{equation} 
  \label{St0} 
  \chi_{\mu}^{\lambda}=0 \qquad \text{ if  }\qquad{\mathcal P}_t(\mu) > {\mathcal H}_t(\lambda)  \text{  for some integer } t. 
  \end{equation} 
  We will call zeros $(\lambda,\mu)$ for which the criterion in ~\eqref{St0} holds \textit{zeros of type III}.

  We became aware of this criterion through a comment of Stanley on MathOverflow~\cite{StanleyMO}.   There, Stanley drew attention to the following elegant corollary of the criterion \eqref{St0}, which may also be found as Corollary 7.5 of \cite{Stanleypaper} or as Exercise 7.60 of \cite{Stanley}: the character value $\chi_{\mu}^{\lambda}$ equals zero if, in the polynomial ring ${\mathbb Z}[x]$,
  \begin{equation}\label{eq:Stanley}
    \prod_{i=1}^{\ell(\mu)}(1-x^{\mu_i})\nmid \prod_{u\in\lambda}(1-x^{h(u)}). 
  \end{equation}
Here $\ell(\mu)$ equals the number of parts in $\mu$, which we denote in descending order by $\mu_1,\dots,\mu_{\ell(\mu)}$ and $h(u)$ denotes the hook length corresponding to the box $u$ in the Young diagram of $\lambda$.  
 
   By considering the roots of both sides of \eqref{eq:Stanley}, one may see the following equivalent formulation of the criterion \eqref{eq:Stanley}:  $\chi_{\mu}^{\lambda}=0$ if there is an integer $t$ such that the number of parts in $\mu$ that are multiples of $t$ exceeds the number of hooks in $\lambda$ that are multiples of $t$.   Since ${\mathcal P}_t(\mu)$ is at least the number of parts in $\mu$ that are multiples of $t$, we see that condition \eqref{eq:Stanley} is a corollary of \eqref{St0}. Further, if $\lambda$ is a $t$-core, then ${\mathcal H}_t(\lambda)=0$, and so if $\mu$ has a part equal to $t$, then \eqref{eq:Stanley} implies that $\chi_{\mu}^{\lambda} =0$.   This reveals that zeros of type III include the zeros of type II previously described (which in turn include the zeros of type I).   
  
  Our main result determines an asymptotic for the number of zeros of types I, II, or III.  It turns out that all three types of zeros have the same leading order asymptotic, so that most zeros of type III are in fact simply zeros of type I.   
  
  \begin{theorem}\label{thm:main} For a natural number $N$, let $Z_I(N)$, $Z_{II}(N)$, and $Z_{III}(N)$ denote the number of pairs $(\lambda,\mu)$ corresponding to zeros of type I, II, and III respectively.   Then for integers $N\ge 3$, all three quantities $Z_I(N)$, $Z_{II}(N)$ and $Z_{III}(N)$ are asymptotically  
 $$ 
   p(N)^2 \Big( \frac{2}{\log N} + O \Big(\frac{(\log \log N)^2}{(\log N)^2}\Big)\Big), 
 $$ 
   where $p(N)$ denotes the number of partitions of $N$.    Moreover, 
   $$ 
   Z_{III}(N) - Z_{II}(N) = O(N^{-\frac 12} p(N)^2). 
   $$  
     \end{theorem}

This result resolves a conjecture of Miller and Scheinerman that most zeros of type II are in fact of type I.   It also addresses a question of Stanley \cite{StanleyMO}, who asked for an understanding of the number of type III zeros, and we see that these are very nearly equal to the number of type II zeros.   It would be possible to give a more accurate asymptotic for $Z_{I}(N)$, with sharper error terms: namely, with an error term $O(N^{-c}p(N)^2)$ for a suitable positive constant $c$, $Z_I(N)$ is approximately
$$ 
p(N)^2 \int_0^1 \exp\Big(-y \Big( 1+ \log \frac{\sqrt{6N}}{\pi y} \Big) \Big) dy.  
$$ 
This can be expressed as an asymptotic expansion 
$$ 
\frac{p(N)^2}{\log (\sqrt{6N}/\pi)} \sum_{k=0}^{\infty} P_k(\log \log \sqrt{6N}/\pi) (\log \sqrt{6N}/\pi)^{-k}, 
$$ 
where $P_k$ is a polynomial of degree $k$.  For small $N$, the effect of the lower order terms in our asymptotic for $Z_I(N)$ is noticeable: for instance, with $N=50000$, the precise asymptotic expansion above suggests that the proportion  of type I zeros is about $0.13$ which is in good agreement with the empirical data in \cite{MS25}, whereas the leading order term in Theorem \ref{thm:main} gives a substantially larger proportion of about $0.1848$.   

The bulk of the difference between $Z_{II}(N)$ and $Z_I(N)$ is comprised of pairs $(\lambda,\mu)$ where $\lambda$ is a $\mu_2$-core but not a $\mu_1$-core.   To determine the asymptotic count of such pairs, we would need to understand how different hook lengths are distributed in a random partition $\lambda$, which we do not know.

    \section{Preliminaries on the asymptotic for $p(N)$}  
 
\noindent The number of conjugacy classes in $S_N$, which equals the number of irreducible representations of $S_N$, is given by the partition function $p(N)$.   In this section we recall some standard facts regarding the Hardy--Ramanujan asymptotic formula for $p(N$), namely 
 $$
 p(N) \sim \frac{\exp(2\pi \sqrt{N}/\sqrt{6})}{4N\sqrt{3}}. 
 $$ 
 Both the asymptotic formula and the salient features of the proof that we now recall will be useful in our subsequent work. 
 
 For $|z|<1$ denote the generating function for $p(N)$ by 
 $$ 
 {\mathcal F}(z) = \prod_{n=1}^{\infty} (1-z^n)^{-1} = \sum_{N=0}^{\infty} p(N) z^N. 
 $$ 
 The Hardy-Ramanujan formula is obtained through an analysis of the Cauchy integral formula applied to $\mathcal{F}(z)$:
 \begin{equation} 
 \label{0} 
 p(N) = \frac{1}{2\pi i} \int_{|z|=q} {\mathcal F}(z) z^{-N-1} dz = \frac{1}{2\pi} \int_{-\pi}^{\pi} {\mathcal F}(q e^{i\theta}) q^{-N} e^{-iN\theta} d\theta,  
 \end{equation} 
 where the parameter $q <1$ is chosen carefully.   The value $q$ is chosen so as to minimize ${\mathcal F}(r) r^{-N}$ over all $r\in (0, 1)$.   Note that 
 $$ 
 \log {\mathcal F}(z) =\sum_{n=1}^{\infty} \log (1-z^n)^{-1} = \sum_{n=1}^{\infty} \sum_{k=1}^{\infty} \frac{z^{nk}}{k} = 
 \sum_{m=1}^{\infty} \frac{\sigma(m)}{m} z^m,
 $$ 
 upon grouping terms according to $m=nk$ and writing $\sigma(m) = \sum_{m=nk} n$. Thus to minimize ${\mathcal F}(r) r^{-N}$, we must choose $r$ so that $r\frac{{\mathcal F}^{\prime}(r)}{{\mathcal F} (r)}  = N$, or in other words, 
 $$ 
 \sum_{m=1}^{\infty} \sigma(m) r^m = N. 
 $$ 
 A suitable value of $q$, which is very nearly the optimal solution to the above equation, is given by 
 \begin{equation}
 \label{1} 
 q = \exp\Big( - \frac{\pi}{\sqrt{6N}}\Big), 
 \end{equation} 
 and throughout this paper when we write $q$ we have this choice in mind.   
 
 For large $x$, a straightforward asymptotic calculation gives 
 \begin{equation}
 \label{1.5} 
 \sum_{n=1}^{\infty} \frac{\sigma(n)}{n} e^{-n/x} = \frac{\pi^2}{6} x - \frac 12 \log (2\pi x) + O(x^{-1}). 
 \end{equation} 
 In fact, one can be much more precise thanks to a modularity relation but~\eqref{2} will be sufficient for our purposes.
 With $q$ as in~\eqref{1} it follows that (taking $x= \sqrt{6N}/\pi$ in \eqref{1.5}) 
  \begin{equation} 
 \label{2} 
 \log {\mathcal F}(q) =\sum_{m=1}^{\infty} \frac{\sigma(m)}{m} \exp\Big(- \frac{\pi m}{\sqrt{6N}} \Big) =  \frac{\pi}{\sqrt{6}}\sqrt{N} - \frac 12 \log (2\sqrt{6N}) +O(N^{-\frac 12}). 
  \end{equation}  
  Returning to the right-most formula in~\eqref{0}, the integrand drops sharply in size as $|\theta|$ increases, and the contribution from $|\theta|$ substantially larger than $N^{-\frac 34}$ is negligible.   To see this, note that 
\begin{align*} 
\log \frac{|{\mathcal F}(q e^{i\theta})|}{{\mathcal F}(q)} &= 
                                                             -\sum_{m=1}^{\infty} \frac{\sigma(m)}{m} (1- \cos m\theta) q^m \\
  &\le - \sum_{m=1}^{\infty} q^{m} (1- \cos m\theta) 
=- \Big( \frac{q}{1-q} - \text{Re} \frac{qe^{-i\theta}}{1-qe^{i\theta}}\Big), 
\end{align*}
and a small calculation (in which it is helpful to distinguish the cases $|\theta|\le (1-q)$ and $\pi \ge |\theta| > (1-q)$) establishes that for an absolute constant $c>0$,
\begin{equation} 
\label{3} 
\log \frac{|{\mathcal F}(q e^{i\theta})|}{{\mathcal F}(q)} \le - c \min ( N^{\frac 32} \theta^2, N^{\frac 12}).  
\end{equation} 

From our analysis above, we conclude that 
\begin{equation} 
\label{4} 
p(N) = \frac{1}{2\pi} \int_{|\theta| \le N^{-\frac 35}} {\mathcal F}(qe^{i\theta}) q^{-N} e^{-iN\theta} d\theta + O\Big( {\mathcal F}(q) q^{-N} \exp(-c N^{\frac 3{10}})\Big). 
\end{equation} 
The choice~\eqref{1} of $q$ is such that, for small $\theta$, the argument of ${\mathcal F}(qe^{i\theta}) e^{-iN\theta}$ is approximately zero and ${\mathcal F}(qe^{i\theta}) e^{-iN\theta}$ is approximated by ${\mathcal F}(q) e^{-C N^{\frac 32} \theta^2}$ for a suitable absolute constant $C>0$.  Evaluating the resulting Gaussian integral leads to the Hardy--Ramanujan formula.  We record two further useful facts following from this analysis: 
\begin{equation} 
\label{5} 
p(N) \sim \frac{{\mathcal F}(q)  q^{-N}}{2 \cdot 6^{\frac 14} \cdot N^{\frac 34}}  \sim \frac{\exp(2\pi \sqrt{N}/\sqrt{6})}{4N\sqrt{3}} 
\end{equation} 
and 
\begin{equation} 
\label{6} 
\int_{-\pi}^{\pi} |{\mathcal F}(qe^{i\theta})| q^{-N} d\theta \ll N^{-\frac 34} {\mathcal F}(q) q^{-N} \ll p(N). 
\end{equation}

 \section{Distribution of hook lengths} 

\noindent For any partition $\lambda$, the \textit{hook} associated to a box $u$ in the Young diagram of $\lambda$ is the collection of boxes to the right of $u$ in the same row as $u$, below $u$ in the same column as $u$, and $u$ itself. The \textit{length} of this hook, denoted by $h(u)$, equals the number of boxes in the hook.

Given a partition $\lambda$ of $N$ and a natural number $t$ we are interested in the quantity ${\mathcal H}_t(\lambda)$, which denotes the number of hook lengths in $\lambda$ that are multiples of $t$.   We will require a result of Han \cite{Han10} that gives a useful formula for the generating function for ${\mathcal H}_t(\lambda)$. If $z$ and $w$ are complex numbers inside the unit circle, Theorem 1.3 of~\cite{Han10} gives 
  \begin{equation} 
  \label{9} 
 \sum_{N=0}^{\infty} \sum_{\lambda \vdash N} z^{N} w^{{\mathcal H}_t(\lambda)} = \prod_{n=1}^{\infty} (1-z^n)^{-1} \frac{(1-z^{t n})^t}{(1-w^n z^{tn})^t}. 
  \end{equation} 
  A partition $\lambda$ of $N$ with ${\mathcal H}_t(\lambda)=0$ (thus no hook lengths that are multiples of $t$) is known as a $t$-core partition, and the number of such partitions is denoted by $c_t(N)$.  Han's formula generalizes the well-known formula for the generating function for $t$-cores, 
  \begin{equation} 
  \label{10} 
  \sum_{N=0}^{\infty} c_t(N) z^N = \prod_{n=1}^{\infty}(1-z^n)^{-1} (1-z^{tn})^t. 
  \end{equation}

In this section, we shall show that ${\mathcal H}_t(\lambda)$ is suitably large for most partitions $\lambda$ in various ranges of $t$. To describe these results fluidly, it is convenient to define the following thresholds for $t$: 
\begin{equation} 
\label{7} T_0 = \frac{\sqrt{6N}}{\pi}(\log N)^{-4}, \qquad T_1 = \frac{\sqrt{6N}}{\pi} \Big( \log \sqrt{N} -20 \Big),
\end{equation} 
and 
\begin{equation} 
\label{8} 
T_2 = \frac{\sqrt{6N}}{\pi} \Big( \log \sqrt{N} + \log \log N - \log \log \log N -20 \Big). 
\end{equation} 
Only the thresholds $T_0$ and $T_1$ are relevant for this section, and the threshold $T_2$ will only appear in Section 5 when we complete the proof of the main theorem.  Throughout we assume that $N$ is sufficiently large, so that quantities like $\log \log \log N$ are well defined.  
 
 \begin{proposition} \label{prop1} Let $N$ be large, and recall that $q= \exp(-\pi/\sqrt{6N})$.   
 
 (i)  If $t\le T_0$ is an integer, then there are $\ll N^{-10} p(N)$ partitions $\lambda$ of $N$ with 
 $$
 {\mathcal H}_t(\lambda) \le \frac{N}{t} \Big(1-\frac 2{\log N}\Big).  
 $$ 
 
 (ii)  In the range $T_0 < t\le T_1$,  the number of partitions $\lambda$ of $N$ with ${\mathcal H}_t(\lambda)  \le \frac{\sqrt{6N}}{4\pi} q^t$ is $\ll N^{-10} p(N)$.  
 \end{proposition} 
 
 \begin{proof} 
 Given an integer $t$, and a parameter $k$ we wish to bound the number of partitions $\lambda$ of $N$ with ${\mathcal H}_t(\lambda) \le k$.   For any $y \in (0,1)$ this is 
 $$ 
 \le y^{-k} \sum_{\lambda \vdash N} y^{{\mathcal H}_t(\lambda)} \le y^{-k} q^{-N} \sum_{\lambda} q^{|\lambda|} y^{{\mathcal H}_t(\lambda)}, 
 $$ 
 which, by Han's formula  \eqref{9} and the asymptotic \eqref{5},  is 
 $$
 \le y^{-k} q^{-N}  {\mathcal F}(q) \prod_{j=1}^{\infty} \frac{(1-q^{tj})^t}{(1-y^j q^{tj})^t} \ll N^{\frac 34} p(N) y^{-k}  \prod_{j=1}^{\infty} \frac{(1-q^{tj})^t}{(1-y^j q^{tj})^t}.   
 $$ 

Now,
$$ 
y^{-k} \prod_{j=1}^{\infty} \frac{(1-q^{tj})^t}{(1-y^j q^{tj})^t} = y^{-k} \exp\Big( t \sum_{j=1}^{\infty} \log \frac{(1-q^{tj})}{(1-y^jq^{tj})}\Big) = y^{-k} \exp\Big(t
\sum_{j=1}^{\infty} \sum_{\ell=1}^{\infty} \frac{(yq^t)^{j\ell} - q^{tj\ell}}{\ell} \Big), 
$$ 
and grouping terms according to $n=j\ell$, we conclude that the number of partitions of $N$ with ${\mathcal H}_t(\lambda) \le k$ is 
 \begin{equation} 
\label{8.5} 
\ll N^{\frac 34} p(N) y^{-k} \exp\Big( -t \sum_{n=1}^{\infty} \frac{\sigma(n)}{n} (1-y^n)q^{tn}\Big). 
\end{equation}

To prove part (i), we choose $y= q^{t/\log N}$, and take $k= Nt^{-1} (1-2/\log N)$.   Using \eqref{1.5} twice (with $x= \sqrt{6N}/(\pi t)$ and $\sqrt{6N}/(\pi t) (1+1/\log N)^{-1}$ both of which are $\gg (\log N)^4$ in the range $t\le T_0$), the bound in \eqref{8.5} is 
$$ 
\ll N^{\frac 34} p(N) \exp\Big( \frac{\pi \sqrt{N}}{\sqrt{6} \log N} \Big(1- \frac{2}{\log N}\Big) - t \Big(\frac{\pi \sqrt{N}}{\sqrt{6} t} \frac{1}{1+\log N} - \frac{1}{\log N} + O((\log N)^{-2})\Big)\Big).
$$
This simplifies to give the bound 
$$ 
\ll N^{\frac 34} p(N) \exp\Big( - \frac{\pi \sqrt{N}}{\sqrt{6} (\log N)^2} + \frac{t}{\log N} + O\Big(\frac{t}{(\log N)^2}\Big)\Big), 
$$ 
which is much smaller than $N^{-10} p(N)$ in the range $t \le T_0$.  

In the range $T_0 < t \le T_1$, we choose $y=q^t$ and $k = \sqrt{6N}q^t/(4\pi)$ in \eqref{8.5}.  Note that $y^{-k} = \exp( t q^t/4)$, and using 
$\sigma(n) \ge n$ and $1-y^n \ge (1-q^t)$ for all $n\ge 1$, 
$$
t\sum_{n=1}^{\infty} \frac{\sigma(n)}{n}(1-y^n) q^{tn} \ge t (1-q^t) \frac{q^t}{1-q^t} = t q^t. 
$$ 
Thus the bound in \eqref{8.5} shows that the number of desired partitions is $\ll N^{\frac 34} p(N) \exp( -\frac 34 tq^t) \ll N^{-10} p(N)$ 
for $T_0 < t \le T_1$.     This completes our proof. 
\end{proof}



In the range $t \ge T_1$, we also need an understanding of $c_t(N)$, the number of $t$-core partitions of $N$, which is provided by the next proposition.   Much more is known about the number of $t$-cores $c_t(N)$, and for a thorough treatment see work of Tyler \cite{Tyler26}.   

\begin{proposition}  \label{prop2}  Let $N$ be large, and recall that $q= \exp(-\pi/\sqrt{6N})$.   If $t \ge T_1$, then the number of partitions $\lambda$ of $N$ with ${\mathcal H}_t(\lambda)=0$ (in other words, $t$-cores $\lambda$) is 
$$ 
c_t(N) = p(N) \Big( \exp(-tq^t) +  O(N^{-\frac 1{12}})\Big). 
$$ 
\end{proposition} 
\begin{proof}  Using the generating function \eqref{10}, we obtain 
$$ 
c_t(N) = \frac{1}{2\pi} \int_{-\pi}^{\pi} {\mathcal F}(qe^{i\theta}) \prod_{n=1}^{\infty} (1- (qe^{i\theta})^{tn})^t q^{-N} e^{-iN\theta} d\theta. 
$$ 
Now for $t\ge T_1$, we have $q^t \ll N^{-\frac 12}$ and $t q^t \ll \log N$, and therefore 
$$ 
 \prod_{n=1}^{\infty} (1- (qe^{i\theta})^{tn})^t = \exp\Big( -t q^t e^{it\theta} + O(tq^{2t}) \Big) = \exp\Big( -tq^t + O( tq^{2t} + tq^t \min(1, t|\theta|))\Big). 
 $$
In view of \eqref{3}, the contribution to $c_t(N)$ from the portion of the integral with $|\theta| \ge N^{-\frac 35}$ is much smaller than $p(N)/N$.  
Restricting now to the portion $|\theta| \le N^{-\frac 35}$, we find (since $t^2 q^t \ll N^{\frac 12 +\epsilon}$ for $t>T_1$) 
$$ 
c_t(N) = \frac{1}{2\pi} \int_{|\theta|\le N^{-\frac 35}} {\mathcal F}(qe^{i\theta}) q^{-N} e^{-iN\theta} 
\exp\Big( - tq^t +O(N^{-\frac 1{12}})\Big) d\theta + O(p(N)/N),  
$$
and the result follows from \eqref{4} and \eqref{6}. 
\end{proof}

\section{Distribution of parts} 

\noindent Given a partition $\mu$ of $N$ and a natural number $t$, recall that 
$$
{\mathcal P}_t(\mu) = \frac{1}{t} \sum_{t| \mu_j} \mu_j, 
$$ 
so that $t{\mathcal P}_t(\mu)$ equals the sum over the parts in $\mu$ that are multiples of $t$.   In suitable ranges of $t$, we shall show that ${\mathcal P}_t(\mu)$ is suitably small for most partitions $\mu$.   Clearly ${\mathcal P}_1(\mu) = N$ for all partitions $\mu$ of $N$, and we may restrict attention to $t\ge 2$.   Much is known about the structure of a random partition (for example, see \cite{Fri93}), but our focus is on quick results that are sufficient for our purposes.

 \begin{proposition}\label{prop3} Let $N$ be large and recall that $q= \exp(-\pi/\sqrt{6N})$.  Uniformly for all $1\le t\le N$, the number of partitions $\mu$ of $N$ with at least one part equal to $t$ is 
 $$ 
 p(N-t)  = q^t p(N) \Big( 1+ O\Big(\min\Big(1, \frac{t}{N^{\frac 34}}\Big) \Big)\Big), 
 $$ 
 and in particular this is always $\ll q^t p(N)$.  
 \end{proposition} 
\begin{proof}  Adjoining a part $t$ to any partition of $N-t$ gives a bijective correspondence with partitions of $N$ having at least one part equal to $t$.   By the integral formula
$$ 
p(N-t) = \frac{1}{2\pi} \int_{-\pi}^{\pi} {\mathcal F}(qe^{i\theta}) q^{-N+t} e^{-iN\theta} e^{it\theta} d\theta = 
\frac{q^{t}}{2\pi} \int_{-\pi}^{\pi} {\mathcal F}(qe^{i\theta}) q^{-N} e^{-iN\theta} \Big( 1+ O( \min(1,t|\theta|)\Big) d\theta. 
$$ 
The main term above equals $q^t p(N)$ and the remainder terms may be bounded using \eqref{3} and \eqref{5}.  The proposition follows.
\end{proof} 

Recall the thresholds $T_0$ and $T_1$ defined in \eqref{7} above.

\begin{proposition}  
\label{prop4} Let $N$ be large and recall that $q= \exp(-\pi/\sqrt{6N})$. 

(i)  In the range $2 \le t\le T_0$, the number of partitions $\mu$ of $N$ with ${\mathcal P}_t(\mu) > \frac{9N}{10t}$ is $\ll N^{-10}p(N)$.

(ii) In the range $T_0 < t\le T_1$, the number of partitions $\mu$ of $N$ with ${\mathcal P}_t(\mu) > \frac{\sqrt{6N}}{4\pi } q^t$ is $\ll N^{-10}p(N)$.  

(iii) In the range $t \ge T_1$, the number of partitions $\mu$ of $N$ with ${\mathcal P}_t(\mu) \ge 2$ is $\ll q^{2t} p(N)$.
\end{proposition}    
\begin{proof}    Let $a_t(n)$ denote the number of partitions of $n$ into parts that are not multiples of $t$.  Separate the parts in a partition of $N$ into the parts that are multiples of $t$, and into the parts that are not multiples of $t$.  The number of partitions $\lambda$ of $N$ with ${\mathcal P}_t(\mu) =\ell$ is clearly $p(\ell) a_t(N-t\ell)$.   Therefore, for any $k$, the number of partitions $\mu$ with ${\mathcal P}_t(\mu) >k$ is 
\begin{equation} 
\label{8.6} \sum_{\ell >k} p(\ell) a_t(N-t\ell) \le \sum_{\ell >k} p(\ell) p(N-t\ell), 
\end{equation} 
since $a_t(n)\le p(n)$ always.  

In the range $2\le t \le T_0$, we use the bound in \eqref{8.6} taking $k= \frac{9N}{10t}$, together with the crude estimate $p(n) \ll \exp(2\pi \sqrt{n}/\sqrt{6})$.  A small calculation shows a bound much smaller than $N^{-10} p(N)$.   A similar argument with $k= \frac{\sqrt{6N}}{4\pi} q^t$ works in the range $T_0 \le t \le T_1$.   

Finally, the same argument gives that for $t\ge T_1$, the number of partitions $\mu$ with ${\mathcal P}_t(\mu) \ge 2$ is (using Proposition \ref{prop3})
$$ 
\ll \sum_{\ell \ge 2} p(\ell) p(N-\ell t) \ll p(N) \sum_{\ell \ge 2} p(\ell) q^{\ell t} \ll q^{2t} p(N). 
$$ 
This completes our proof. 
\end{proof} 

\section{Proof of the main theorem} 

\noindent We will now count the number of pairs $(\lambda, \mu)$ of partitions of $N$ satisfying Stanley's criterion \eqref{St0}.   Equivalently, these are the pairs $(\lambda,\mu)$ for which there exists a natural number $t$ with ${\mathcal P}_t(\mu) > {\mathcal H}_t(\lambda)$.  Since ${\mathcal H}_1(\lambda)={\mathcal P}_1(\mu)=N$, we may assume that $t \ge 2$.   There are several different ways in which this could happen, and we consider the following possibilities: 

\begin{enumerate}[(i)]
\item For some $2\le t\le T_1$ we have ${\mathcal P}_t (\mu) > {\mathcal H}_t(\lambda)$. 
\item For some $t>T_1$ one has ${\mathcal P}_t(\mu) \ge 2$. 
\item There exists $T_1 <t \le T_2$ with ${\mathcal P}_t(\mu)=1$ and ${\mathcal H}_t(\lambda)=0$.   
\item The partition $\mu$ contains two parts $t_1 < t_2$ that are both larger than $T_2$. 
\item The partition $\mu$ has a unique part $t >T_2$ and $\lambda$ is a $t$-core.  
\end{enumerate} 

Any pair $(\lambda,\mu)$ satisfying Stanley's criterion must satisfy at least one of these five conditions above.    We will show that the number of pairs $(\lambda,\mu)$ satisfying either of the conditions (i) or (ii) above is 
\begin{equation} 
\label{13} 
\ll N^{-\frac 12} p(N)^2, 
\end{equation} 
while the number of pairs $(\lambda,\mu)$ satisfying either condition (iii) or (iv) is 
\begin{equation} 
\label{13.5} 
\ll p(N)^2 \frac{(\log \log N)^2}{(\log N)^2}, 
\end{equation} 
and lastly the number of pairs satisfying condition (v) equals 
\begin{equation} 
\label{14} p(N)^2 \Big( \frac{2}{\log N} + O\Big(\frac{\log \log N}{(\log N)^2}\Big)\Big). 
\end{equation} 
Since every pair satisfying condition (v) corresponds to a zero of type I, and zeros of type III that are not of type II must satisfy either condition (i) or (ii), our main Theorem would follow at once.

We begin with bounding the pairs satisfying condition (i).   Given an integer  $t \le T_0$, if ${\mathcal P}_t(\mu) > {\mathcal H}_t(\lambda)$ then either ${\mathcal H}_t(\lambda) \le (N/t)(1-2/\log N)$ or ${\mathcal P}_t(\mu) > 9N/(10 t)$.   By Propositions~\ref{prop1} and~\ref{prop4}, the number of such pairs $(\lambda,\mu)$ is $\ll p(N)^2 N^{-10}$.   Summing this over all the possibilities for $t$, we conclude that there are at most $\ll N^{-9} p(N)^2$ pairs $(\lambda,\mu)$ with ${\mathcal P}_t(\mu)> {\mathcal H}_t(\lambda)$ for some $2\le t\le T_0$. Similarly, Propositions~\ref{prop1} and~\ref{prop4} show that there are at most $\ll N^{-9} p(N)^2$ pairs $(\lambda, \mu)$ with ${\mathcal P}_t(\mu) > \frac{\sqrt{6N}}{4\pi} q^t$ or ${\mathcal H}_t(\lambda)\le \frac{\sqrt{6N}}{4\pi} q^t$ for some $T_0 \le t\le T_1$.  We thus conclude that the number of pairs satisfying condition (i) is much smaller than the bound claimed in \eqref{13}.

We move now to condition (ii).   By part (iii) of Proposition~\ref{prop4}, the number of partitions $\mu$ with ${\mathcal P}_t(\mu) \ge 2$ for some $t>T_1$ is 
$$ 
\ll \sum_{t>T_1} q^{2t} p(N) \ll \frac{q^{2T_1}}{1-q^2}  p(N) \ll \sqrt{N} q^{2T_1} p(N) \ll N^{-\frac 12} p(N). 
$$ 
Thus there are $\ll N^{-\frac 12} p(N)^2$ pairs of partitions satisfying condition (ii), which is the bound in \eqref{13}.

Now consider condition (iii), and suppose that $t$ is in the range  $T_1 < t\le T_2$.  By Propositions~\ref{prop2} and~\ref{prop3}, the number of pairs $(\lambda, \mu)$ with 
${\mathcal P}_t(\mu) =1$ and ${\mathcal H}_t(\lambda)=0$ for some $t$ in this range is 
\begin{align*}
&\ll p(N)^2 \sum_{T_1 < t\le T_2} q^t \exp(-tq^t) \ll p(N)^2 \exp(-T_2 q^{T_2})  \sum_{T_1 < t\le T_2} q^t 
\\
& \ll \frac{p(N)^2}{(\log N)^{100}} \sqrt{N} q^{T_1} \ll \frac{p(N)^2}{(\log N)^{100}}.  
\end{align*}
 This is negligible in comparison to \eqref{13.5}.

Observe next that the number of partitions $\mu$ with two parts $t_1$ and $t_2$ both larger than $T_2$ is 
$$ 
\ll \sum_{t_2 > t_1 > T_2} p(N-t_1-t_2) \ll p(N) \sum_{t_2 > t_2 > T_2} q^{t_1 + t_2} \ll p(N) \Big( \frac{q^{T_2}}{1-q}\Big)^2 \ll p(N) \frac{(\log \log N)^2}{(\log N)^2}. 
$$ 
There are thus at most $\ll p(N)^2 (\log \log N)^2/(\log N)^2$ pairs satisfying condition (iv).  

Finally we are left with case (v).   Here we must count partitions $\mu$ that have a unique part $t > T_2$ (so that ${\mathcal P}_t(\mu)=1$) and partitions $\lambda$ that are $t$-cores.  Since case (iv) reveals that there are few partitions $\mu$ with two parts larger than $T_2$, we may ignore the uniqueness condition, and focus simply on partitions $\mu$ containing a part $t>T_2$ and corresponding $t$-cores $\lambda$. By Propositions~\ref{prop2} and~\ref{prop3} the number of such pairs is 
\begin{align*}
&p(N)^2 \sum_{t > T_2} \Big( \exp(-tq^t) + O( N^{-\frac {1}{12}} ) \Big) q^t \Big( 1+ O\Big( \min \Big(1, \frac{t}{N^{\frac 34}}\Big) \Big) \\
= &p(N)^2 
\Big( \sum_{t> T_2} q^t \exp(-tq^t) + O(N^{-\frac 1{12}})\Big). 
\end{align*}
To evaluate the sum above, note that for $t>T_2$ 
\begin{align*} 
q^t \exp(-tq^t) &= \int_{t}^{t+1} q^x \exp(-xq^x) dx + O\Big( \max_{t\le x \le t+1} \Big| (q^x \exp(-xq^x))^{\prime} \Big| \Big) 
\\
&= \int_{t}^{t+1} q^x \exp(-xq^x) dx + O(q^t N^{-\frac 12}), 
\end{align*}
so that 
\begin{align*}
\sum_{t>T_2} q^t \exp(-tq^t) &= \int_{T_2}^{\infty} q^x \exp(-xq^x) dx + O\Big( \sum_{t>T_2} q^t N^{-\frac 12}\Big) \\
&= \int_{0}^{\infty} q^{T_2+y} \exp(-(T_2+y) q^{T_2+y}) dy + O(N^{-\frac 12}). \\ 
\end{align*}
Now for $y\ge 0$ (using $e^{-\xi} = 1+ O(\xi)$ for $\xi >0$ and that $yq^y \le 1/\log (1/q)$ for $y>0$)  
\begin{align*}
\exp(-(T_2 +y) q^{T_2+y}) &= \exp(-T_2 q^{T_2+y}) \Big( 1+ O(yq^{T_2+y}) \Big) \\
                          &= \exp (-T_2 q^{T_2+y})  \Big( 1+ O\Big(\frac{q^{T_2}}{\log (1/q)}\Big)\Big)\\
  &=  \exp (-T_2 q^{T_2+y}) \Big( 1+ O\Big(\frac{\log \log N}{\log N}\Big)\Big), 
\end{align*}
and 
$$ 
\int_0^{\infty} q^{T_2+y} \exp(-T_2 q^{T_2+y}) dy = \frac{1-\exp(-T_2 q^{T_2})}{T_2 \log (1/q)} = \frac{2}{\log N} + O\Big(\frac{\log \log N}{(\log N)^2}\Big). 
$$
 Thus the count of pairs satisfying condition (v) obeys the asymptotic \eqref{14}, and the proof of Theorem 1 is complete.

\bibliographystyle{plain} 
\bibliography{bib}

  \end{document}